\documentclass[10pt]{amsart}
\usepackage{graphicx}
\usepackage{amssymb}
\usepackage{amsmath}
\usepackage{amsthm}
\usepackage{amscd}
\usepackage[all]{xy}

\newtheorem{thm}{Theorem}[section]

\newtheorem{cor}[thm]{Corollary}
\newtheorem{lem}[thm]{Lemma}

\newtheorem{prop}[thm]{Proposition}
\theoremstyle{definition}
\newtheorem{defn}[thm]{Definition}
\theoremstyle{remark}
\newtheorem{rem}[thm]{\bf Remark}
\numberwithin{equation}{section}

\begin{document}
\title[Homotopy Equivalences induced by Balanced Pairs]{Homotopy Equivalences induced by Balanced Pairs}
\author[  Xiao-Wu Chen
] { Xiao-Wu Chen}
\thanks{This project was supported by China Postdoctoral Science Foundation
(No.s 20070420125 and 200801230) and by National Natural Science
Foundation of China (No.10971206).}
\date{Nov. 28, 2009}
\thanks{E-mail: xwchen$\symbol{64}$mail.ustc.edu.cn, URL: http://math.ustc.edu.cn/$\sim$xwchen}

\maketitle

\dedicatory{}%
\commby{}%

\begin{abstract}
We introduce the notion of balanced pair of additive subcategories
in an abelian category. We give sufficient conditions under which
the balanced pair of subcategories gives rise to equivalent homotopy
categories of complexes. As an application, we prove that for a
left-Gorenstein ring, there exists a triangle-equivalence between
the homotopy category of its Gorenstein projective modules and the
homotopy category of its Gorenstein injective modules, which
restricts to a triangle-equivalence between the homotopy category of
projective modules and the homotopy category of injective modules.
In the case of commutative Gorenstein rings we prove that up to a
natural isomorphism our equivalence extends Iyengar-Krause's
equivalence.
\end{abstract}

\section{Introduction and Main Results}

 Let $\mathcal{A}$ be an abelian category. Let $\mathcal{X}\subseteq \mathcal{A}$
 be a full additive subcategory which is closed under taking direct summands.
 Let $M\in \mathcal{A}$. A morphism $\theta \colon X \rightarrow M$ is
called a \emph{right $\mathcal{X}$-approximation} of $M$, if $X \in
\mathcal{X}$ and any morphism from an object in $\mathcal{X}$ to $M$
factors through $\theta$. The subcategory $\mathcal{X}$ is called
\emph{contravariantly finite} (= \emph{precovering}) if each object
in $\mathcal{A}$ has a right $\mathcal{X}$-approximation (see
\cite[p.81]{AS1} and \cite[Definition 1.1]{EJ1}).

\vskip 5pt

Recall that for a contravariantly finite subcategory
$\mathcal{X}\subseteq \mathcal{A}$ and an object $M \in \mathcal{A}$
an \emph{$\mathcal{X}$-resolution} of $M$ is  a complex $\cdots
\rightarrow  X^{-2} \stackrel{d^{-2}}\rightarrow X^{-1}
\stackrel{d^{-1}}\rightarrow X^0 \stackrel{\varepsilon} \rightarrow
M \rightarrow 0 $ with each $X^i\in \mathcal{X}$ such that it is
acyclic  by applying the functor ${\rm Hom}_\mathcal{A}(X, -)$ for
each $X \in \mathcal{X}$; this is equivalent to that each induced
morphism $X^{-n}\rightarrow {\rm Ker}d^{-n+1}$ is a right
$\mathcal{X}$-approximation. Here we identify $M$ with ${\rm
Ker}d^{1}$ and $\varepsilon$ with $d^0$. We denote sometimes the
$\mathcal{X}$-resolution by
$X^\bullet\stackrel{\varepsilon}\rightarrow M$ where $X^\bullet=
\cdots \rightarrow  X^{-2} \stackrel{d^{-2}}\rightarrow X^{-1}
\stackrel{d^{-1}}\rightarrow X^0 \rightarrow 0$ is the \emph{deleted
$\mathcal{X}$-resolution} of $M$.  Note that by a version of
Comparison Theorem,  the $\mathcal{X}$-resolution is unique up to
homotopy (\cite[p.169, Ex.2]{EJ}). Recall that the
\emph{$\mathcal{X}$-resolution dimension} $\mathcal{X} {\rm
\mbox{-}res.dim}\; M$ of an object $M$ is defined to be the minimal
integer $n\geq 0$ such that there is an $\mathcal{X}$-resolution
$0\rightarrow X^{-n} \rightarrow \cdots \rightarrow X^0 \rightarrow
M \rightarrow 0$. If there is no such an integer, we set
$\mathcal{X} {\rm \mbox{-}res.dim}\; M=\infty$. Define the
\emph{global $\mathcal{X}$-resolution dimension}
$\mathcal{X}\mbox{-res.dim}\; \mathcal{A}$ to be the supreme of the
$\mathcal{X}$-resolution dimensions of all the objects in
$\mathcal{A}$.

\vskip 5pt

 Let $\mathcal{Y}\subseteq \mathcal{A}$ be another full additive
subcategory which is closed under taking direct summands. Dually one
has the notion of \emph{left $\mathcal{Y}$-approximation} and then
the notions of \emph{covariantly finite subcategory},
\emph{$\mathcal{Y}$-coresolution} and
\emph{$\mathcal{Y}$-coresolution dimension}
$\mathcal{Y}\mbox{-cores.dim}\; N$ of an object
 $N$; furthermore, one has the notion of \emph{global $\mathcal{Y}$-coresolution dimension}
$\mathcal{Y}\mbox{-cores.dim}\; \mathcal{A}$. For details, see
\cite[Section 2]{Bel} and \cite[8.4]{EJ}. \vskip 5pt

 Inspired by \cite[Definition 8.2.13]{EJ}, we introduce the following notion.

\begin{defn}
A pair $(\mathcal{X}, \mathcal{Y})$  of additive subcategories  in
$\mathcal{A}$ is called a \emph{balanced pair} if the following
conditions are satisfied:
\begin{enumerate}
\item[(BP0)] the subcategory $\mathcal{X}$ is  contravariantly
finite and $\mathcal{Y}$ is covariantly finite;
\item[(BP1)] for each object $M$, there is an $\mathcal{X}$-resolution $X^\bullet
\rightarrow  M$ such that it is acyclic by applying the functors
${\rm Hom}_\mathcal{A}(-, Y)$ for all $Y\in \mathcal{Y}$;
\item[(BP2)]  for each object $N$, there is a
$\mathcal{Y}$-coresolution $ N \rightarrow Y^\bullet$ such that it
is acyclic by applying the functors ${\rm Hom}_\mathcal{A}(X, -)$
for all $X\in \mathcal{X}$.
\end{enumerate}
\end{defn}

Balanced pairs enjoy certain ``balanced" property; see Lemma
\ref{lem:balanced}. As mentioned above, the $\mathcal{X}$-resolution
of an object $M$ is unique up to homotopy. Hence the condition (BP1)
may be rephrased as: any $\mathcal{X}$-resolution of $M$ is acyclic
by applying the functors ${\rm Hom}_\mathcal{A}(-, Y)$ for all $Y\in
\mathcal{Y}$. Similar remarks hold for (BP2). Balanced pairs arise
naturally from cotorsion triples; see Proposition
\ref{prop:blanced}.

\vskip 5pt

We say that a contravariantly finite subcategory
$\mathcal{X}\subseteq \mathcal{A}$ is \emph{admissible} if each
right $\mathcal{X}$-approximation is epic. Dually one has the notion
of \emph{coadmissible} covariantly finite subcategory. It turns out
that  for  a balanced pair $(\mathcal{X}, \mathcal{Y})$,
$\mathcal{X}$ is admissible if and only if $\mathcal{Y}$ is
coadmissible; see Corollary \ref{cor:admissible}. In this case, we
say that the balanced pair is \emph{admissible}. Moreover, for an
admissible balanced pair $(\mathcal{X}, \mathcal{Y})$,
$\mathcal{X}\mbox{-res.dim}\;
\mathcal{A}=\mathcal{Y}\mbox{-cores.dim}\; \mathcal{A}$; see
Corollary \ref{cor:finitedimension}. If both the dimensions are
finite, we say that the balanced pair is \emph{of finite dimension}.

\vskip 5pt

For an additive category $\mathfrak{a}$, denote by
$\mathbf{K}(\mathfrak{a})$ the homotopy category of complexes in
$\mathfrak{a}$. Our  main result is as follows. It gives sufficient
conditions under which a balanced pair of subcategories gives rise
to equivalent homotopy categories of complexes.

\vskip 5pt

\noindent {\bf Theorem A.}\; \emph{Let $(\mathcal{X},\mathcal{Y})$
be a balanced pair of additive subcategories in an abelian category
$\mathcal{A}$ which is admissible and of finite dimension. Then
there is a triangle-equivalence $\mathbf{K}(\mathcal{X})\simeq
\mathbf{K}(\mathcal{Y})$.}

\vskip 5pt

The proof of Theorem A makes use of the notion of relative derived
category; see Definition \ref{defn:relativedc} and compare
\cite{Ne1,Bu}. In Section 3, we study the relation between homotopy
categories and relative derived categories.

\vskip 5pt

 Our second result is an application of Theorem A to
Gorenstein homological algebra.

\vskip 5pt

 Let $R$ be a  ring with identity. Denote by
$R\mbox{-Mod}$ the category of (left) $R$-modules, and by
$R\mbox{-Proj}$ (resp. $R\mbox{-Inj}$, $R\mbox{-mod}$) the full
subcategory consisting of projective (resp. injective, finitely
presented) $R$-modules. Recall from \cite[p.400]{AM} that a complex
$P^\bullet$ of projective modules is \emph{totally-acyclic} if it is
acyclic and for any projective module $Q$ the Hom complex ${\rm
Hom}_R(P^\bullet, Q)$ is acyclic (also see \cite{Hol,IK}). Following
\cite{EJ1, EJ} a module $G$ is called \emph{Gorenstein projective}
if there is a totally-acyclic complex $P^\bullet$ such that the
zeroth cocycle $Z^0(P^\bullet)$ is isomorphic to $G$, in which case
the complex $P^\bullet$ is said to be a \emph{complete resolution}
of $G$. Denote by $R\mbox{-GProj}$ the full subcategory  of
$R\mbox{-Mod}$ consisting of Gorenstein projective modules. Note
that $R\mbox{-Proj}\subseteq R\mbox{-GProj}$. Dually one has the
full subcategory $R\mbox{-GInj}$ of $R\mbox{-Mod}$ consisting of
\emph{Gorenstein injective} modules and observes that $R\mbox{-Inj}
\subseteq R\mbox{-GInj}$.

\vskip 5pt

Recall that a ring $R$ is \emph{Gorenstein} if it is two-sided
noetherian and the regular module $R$ has finite injective dimension
on both sides. Following \cite{Bel} a ring $R$ is
\emph{left-Gorenstein} provided that any module in $R\mbox{-Mod}$
has finite projective dimension if and only if it has finite
injective dimension. Note that Gorenstein rings are left-Gorenstein
(by \cite[Corollary 6.11]{Bel} or \cite[Chapter 9]{EJ}), while the
converse is not true in general (see \cite{EEGI}).

\vskip 5pt

 Our second result is as
follows, the proof of which makes use of a characterization theorem
of left-Gorenstein rings by Beligiannis (\cite{Bel}; compare a
recent work by Enochs, Estrada and Garc\'{\i}a Rozas  on cotorsion
pairs on Gorenstein categories \cite{EEG}).

\vskip 5pt

 \noindent {\bf Theorem B.}\; \emph{Let $R$  be a left-Gorenstein ring. Then we have a
triangle-equivalence $\mathbf{K}(R\mbox{-{\rm GProj}}) \simeq
\mathbf{K}(R\mbox{-{\rm GInj}})$, which restricts to a
triangle-equivalence $\mathbf{K}(R\mbox{-{\rm Proj}}) \simeq
\mathbf{K}(R\mbox{-{\rm Inj}})$.}

\vskip 5pt

   Theorem B is related to a recent result by Iyengar and Krause (\cite{IK}).
 In their paper, they prove that for a ring $R$ with a dualizing complex,
 in particular a commutative Gorenstein ring,
 there is a triangle-equivalence $\mathbf{K}(R\mbox{-Proj}) \simeq \mathbf{K}(R\mbox{-Inj})$  which
 is given by tensoring with the dualizing complex;
 see \cite[Theorem 4.2]{IK}. We will
  refer to the equivalence as \emph{Iyengar-Krause's equivalence}.
 In the case of commutative Gorenstein rings, we compare the
 equivalences in Theorem B with  Iyengar-Krause's equivalence. It
 turns out that up to a natural isomorphism the first equivalence in
 Theorem B extends Iyengar-Krause's equivalence;
 see Proposition \ref{prop:comparisonequivalence}.

\vskip 5pt

We draw an immediate consequence of Theorem B. For a triangulated
category $\mathcal{T}$ with arbitrary coproducts, denote by
$\mathcal{T}^c$ the full subcategory of its compact objects
(\cite{Ne}). For an abelian category $\mathcal{A}$, denote by
$\mathbf{D}^b(\mathcal{A})$ its bounded derived category. We denote
by $R^{\rm op}$ the opposite ring of a ring $R$.

\vskip 5pt

\noindent {\bf Corollary C.}\; \emph{Let $R$ be a left-Gorenstein
ring which is left noetherian and right coherent. Then there is a
duality $\mathbf{D}^b(R^{\rm op}\mbox{-{\rm mod}}) \simeq
\mathbf{D}^b(R\mbox{-{\rm mod}})$ of triangulated categories.}

\begin{proof} We apply Theorem B to get a triangle-equivalence
$\mathbf{K}(R\mbox{-{\rm Proj}}) \simeq \mathbf{K}(R\mbox{-{\rm
Inj}})$. Note that there are  natural identifications
$\mathbf{K}(R\mbox{-Proj})^c \simeq \mathbf{D}^b(R^{\rm
op}\mbox{-mod})^{\rm op}$ by Neeman (\cite[Proposition 7.12]{Ne2};
compare J{\o}rgensen \cite[Theorem 3.2]{J}), and
$\mathbf{K}(R\mbox{-Inj})^c\simeq \mathbf{D}^b(R\mbox{-mod})$ by
Krause (\cite[Proposition 2.3(2)]{Kr}). Finally observe that a
triangle-equivalence restricts to a triangle-equivalence between the
full subcategories of compact objects.
\end{proof}

Let us remark that the duality above for commutative Gorenstein
rings, more generally, for rings with dualizing complexes,  is well
known (compare \cite[Chapter V]{Har} and \cite[Proposition
3.4(2)]{IK}). It is closely  related to Grothendieck's duality
theory; see \cite[Section 2]{Ne2}.

\vskip 5pt

 We fix some notation. Recall that a complex $X^\bullet=(X^n, d_X^n)_{n\in \mathbb{Z}}$
  in  an additive category $\mathfrak{a}$ is
  a sequence $X^n$ of objects  together with differentials $d_X^n\colon  X^n\rightarrow
X^{n+1}$ such that $d_X^{n+1}\circ d_X^n=0$; a chain map $f^\bullet
\colon X^\bullet \rightarrow Y^\bullet $ between complexes consists
of morphisms $f^n\colon X^n\rightarrow Y^n$ which commute with the
differentials. Denote by $\mathbf{C}(\mathfrak{a})$ the category of
complexes in $\mathfrak{a}$ and by $\mathbf{K}(\mathfrak{a})$ the
homotopy category; denote by $[1]$ the \emph{shift functor} on both
$\mathbf{C}(\mathfrak{a})$ and $\mathbf{K}(\mathfrak{a})$  which is
defined by $(X^\bullet[1])^n=X^{n+1}$and
$d^{n}_{X[1]}=(-1)d_X^{n+1}$. Recall that the \emph{mapping cone}
${\rm Cone}(f^\bullet)$ of a chain map $f^\bullet:
X^\bullet\rightarrow Y^\bullet$  is a complex such that ${\rm
Cone}(f^\bullet)^n=X^{n+1}\oplus Y^n$ and $d^n_{{\rm
Cone}(f^\bullet)}=\begin{pmatrix}-d_X^{n+1} & 0\\
                                   f^{n+1} & d_Y^n\end{pmatrix}$.
We have a \emph{distinguished triangle}  $X^\bullet
\stackrel{f^\bullet}\rightarrow Y^\bullet \stackrel{{0\choose
1}}\rightarrow {\rm Cone}(f^\bullet) \stackrel{(1\; 0)}\rightarrow
X^\bullet[1]$ in $\mathbf{K}(\mathfrak{a})$ associated to the chain
map $f^\bullet$. We also need the \emph{degree-shift functor} $(1)$
on complexes defined by $(X^\bullet(1))^{n}=X^{n+1}$ and
$d^n_{X(1)}=d_X^{n+1}$. Denote by $(r)$ the $r$-th power of the
functor $(1)$. For a complex $X^\bullet$ in an abelian category,
denote by $H^n(X^\bullet)$ the $n$-th cohomology. For more on
homotopy categories and triangulated categories, we refer to
\cite{V1, Har, Iv, Ha1, GM1, Ne}.

\section{Balanced Pair and Cotorsion Triple}
In this section, we will study various properties of balanced pairs
of subcategories in an abelian category. Balanced pairs arise
naturally from cotorsion triples, while the latter are closely
related to the notion of cotorsion pair (\cite{Hov,EEG}).

\vskip 5pt

Let $\mathcal{A}$ be an abelian category. Let us emphasize that in
what follows all subcategories in $\mathcal{A}$ are full additive
subcategories closed under taking direct summands. Recall that we
have introduced the notion of balanced pair of subcategories in
Section 1. The following ``balanced" property of a balanced pair
justifies the terminology.

\begin{lem}\label{lem:balanced}
Let $(\mathcal{X}, \mathcal{Y})$ be a balanced pair of subcategories
in $\mathcal{A}$.  Let $M, N \in \mathcal{A}$ with an
$\mathcal{X}$-resolution $X^\bullet \rightarrow M$ and a
 $\mathcal{Y}$-coresolution $N \rightarrow Y^\bullet$. Then for each $n\geq 0$
 there exists a natural isomorphism
 $$H^n({\rm Hom}_\mathcal{A}(X^\bullet, N))\simeq H^n({\rm Hom}_\mathcal{A}(M,
 Y^\bullet)).$$ \end{lem}

\begin{proof}
The result can be proven similarly as  \cite[Theorem 8.2.14]{EJ}.
One can also prove it by considering the two collapsing spectral
sequences associated
 to the Hom bicomplex ${\rm Hom}_\mathcal{A}(X^\bullet, Y^\bullet)$
 as in the classical homological algebra (\cite[Chapter XVI, Section 1]{CE}).
\end{proof}

Let $\mathcal{X}\subseteq \mathcal{A}$ be a subcategory. Let
$Z^\bullet$ be a complex in $\mathcal{A}$. We say that $Z^\bullet$
is \emph{right $\mathcal{X}$-acyclic} provided that the Hom
complexes ${\rm Hom}_\mathcal{A}(X, Z^\bullet)$ are acyclic for all
$X\in \mathcal{X}$. Dually we have the notion of \emph{left
$\mathcal{Y}$-acyclic complex}. \vskip 5pt

 The following observation is useful.

\begin{prop}\label{firstprop}
Let $\mathcal{A}$ be an abelian category, and let $\mathcal{X}$
(resp. $\mathcal{Y}$) be a contravariantly finite (resp. covariantly
finite) subcategory. Then the pair $(\mathcal{X}, \mathcal{Y})$ is
balanced if and only if the class of right $\mathcal{X}$-acyclic
complexes coincides with the class of left $\mathcal{Y}$-acyclic
complexes.
\end{prop}

\begin{proof} Note that an $\mathcal{X}$-resolution is right
$\mathcal{X}$-acyclic and the condition (BP1) says that an
$\mathcal{X}$-resolution is left $\mathcal{Y}$-acyclic. Dual remarks
hold for (BP2). Thus the ``if" part follows immediately.

To see the ``only if" part, assume that the pair $(\mathcal{X},
\mathcal{Y})$ is balanced. We only show that right
$\mathcal{X}$-acyclic complexes are left $\mathcal{Y}$-acyclic and
leave the dual part to the reader. Assume that $Z^\bullet=(Z^n,
d_Z^n)_{n\in \mathbb{Z}}$ is a complex. Consider the induced
``short" complexes $0\rightarrow {\rm Ker}d_Z^{n} \rightarrow Z^n
\rightarrow {\rm Ker}d_Z^{n+1} \rightarrow 0$ for all $n\in
\mathbb{Z}$. Let us remark that such ``short" complexes are  left
exact sequences, and that they are not necessarily short exact
sequences. Since $Z^\bullet$ is right $\mathcal{X}$-acyclic, all the
induced ``short" complexes are right $\mathcal{X}$-acyclic. Observe
that if all the induced  ``short" complexes are left
$\mathcal{Y}$-acyclic, then so is $Z^\bullet$. Therefore it suffices
to show that a left short exact sequence which is right
$\mathcal{X}$-acyclic is necessarily left $\mathcal{Y}$-acyclic.

  Let $0
\rightarrow M' \rightarrow M \rightarrow M'' \rightarrow 0$ be a
left exact sequence which is right $\mathcal{X}$-acyclic. We will
show that it is left $\mathcal{Y}$-acyclic. Choose
$\mathcal{X}$-resolutions $X'^\bullet \rightarrow M'$ and
$X''^\bullet \rightarrow M''$. By a version of Horseshoe Lemma
(\cite[Lemma 8.2.1]{EJ}), we have a commutative diagram

\[\xymatrix{
0\ar[r] & X'^\bullet \ar[r]^{{1\choose 0}} \ar[d] & X^\bullet
\ar[r]^{(0\;  1)} \ar[d] &
X''^\bullet \ar[d]\ar[r] & 0 \\
0 \ar[r] & M'  \ar[r] & M \ar[r] & M'' \ar[r] & 0}\]
 where the
complex $X^\bullet$ satisfies that for each $n \in \mathbb{Z}$,
$X^n=X'^n\oplus X''^n$ and that the middle column is an
$\mathcal{X}$-resolution. Then  for each $Y\in \mathcal{Y}$ we have
a commutative diagram of abelian groups

\[\xymatrix{
0 \ar[r] & {\rm Hom}_\mathcal{A}(M'', Y) \ar[d]\ar[r] &  {\rm
Hom}_\mathcal{A}(M, Y) \ar[d]\ar[r] & {\rm Hom}_\mathcal{A}(M', Y)
\ar[r] \ar[d]
& 0\\
 0\ar[r] & {\rm Hom}_\mathcal{A}(X''^\bullet, Y) \ar[r]^{{0\choose 1}}  & {\rm Hom}_\mathcal{A}(X^\bullet, Y)
 \ar[r]^{(1\; 0)}  & {\rm Hom}_\mathcal{A}(X'^\bullet, Y)
\ar[r] & 0 }\]  By (BP1) each  column is an acyclic complex. The
bottom row is a sequence of complexes, every degree of which is a
split short exact sequence. We infer that the upper row is exact by
the homology exact sequence. Therefore we deduce that  $0
\rightarrow M' \rightarrow M \rightarrow M'' \rightarrow 0$ is left
$\mathcal{Y}$-acyclic, as required.
\end{proof}

\vskip 5pt

Recall that a  contravariantly finite subcategory $\mathcal{X}$ of
$\mathcal{A}$ is admissible provided that each right
$\mathcal{X}$-approximation is epic. It is equivalent to that any
right $\mathcal{X}$-acyclic complex is indeed acyclic. Similar
remarks hold for coadmissible covariantly finite subcategories. Then
we observe the following  direct consequence of Proposition
\ref{firstprop}.

\begin{cor}\label{cor:admissible}
Let $(\mathcal{X}, \mathcal{Y})$ be a balanced pair. Then
$\mathcal{X}$ is admissible if and only if $\mathcal{Y}$ is
coadmissible. \hfill $\square$
\end{cor}

Recall that in the case of the corollary above, we say that the
balanced pair $(\mathcal{X}, \mathcal{Y})$ is admissible.

\vskip 5pt

The following result on resolution dimensions is well known. However
it seems that there are no precise references. We include here a
proof.

\begin{lem} \label{homologydimension} Let $\mathcal{X}\subseteq \mathcal{A}$ be
a contravariantly finite subcategory. Let $M\in \mathcal{A}$ and let
$n_0\geq 0$. Assume that $\mathcal{X}$ is admissible. The following
statements are equivalent: \begin{enumerate}
\item $\mathcal{X}\mbox{-res.dim}\; M \leq n_0$;
\item for each $\mathcal{X}$-resolution $X^\bullet
\rightarrow M$ and each object $N$,  $H^n({\rm
Hom}_\mathcal{A}(X^\bullet, N))=0$ for all $n> n_0$;
\item for each $\mathcal{X}$-resolution $X^\bullet \rightarrow M$ with
$X^\bullet=(X^{-n}, d_X^{-n})_{n\geq 0}$, the object ${\rm
Ker}d_X^{-n_0+1}$ belongs to $\mathcal{X}$.
\end{enumerate}
\end{lem}

\begin{proof} For ``$(1) \Rightarrow (2)$", choose an
$\mathcal{X}$-resolution  $X_0^\bullet \rightarrow M$ such that
$X_0^{-n}=0$ for $n> n_0$. Then $X_0^\bullet$ and $X^\bullet$ are
homotopically equivalent, thus so are the Hom complexes ${\rm
Hom}_\mathcal{A}(X_0^\bullet, N)$ and ${\rm
Hom}_\mathcal{A}(X^\bullet, N)$. Hence for each $n$ we have
$H^n({\rm Hom}_\mathcal{A}(X_0^\bullet, N))\simeq H^n({\rm
Hom}_\mathcal{A}(X^\bullet, N))$. Then (2) follows directly.

\vskip 3pt

 For ``$(2) \Rightarrow (3)$", note that $H^{n_0+1}({\rm
Hom}_\mathcal{A}(X^\bullet, {\rm Ker}d_X^{-n_0}))=0$ implies that
the naturally induced morphism $\bar{d}\colon X^{-n_0-1} \rightarrow
{\rm Ker}d_X^{-n_0}$ factors through $d_X^{-n_0-1}$, say there is a
morphism $\pi\colon  X^{-n_0} \rightarrow {\rm Ker}d_X^{-n_0+1}$
such that $\bar{d}=\pi\circ d_X^{-n_0-1}$. Observe that
$d_X^{-n_0-1}={\rm inc}\circ \bar{d}$, where ``${\rm inc}$" is the
inclusion morphism of ${\rm Ker}d_X^{-n_0}$ into $X^{-n_0}$. Then we
have $\bar{d}=(\pi\circ {\rm inc})\circ\bar{d}$. Note that $\bar{d}$
is a right $\mathcal{X}$-approximation and that $\mathcal{X}$ is
admissible. Hence $\bar{d}$ is epic, and then  $\pi\circ {\rm
inc}={\rm Id}_{{\rm Ker}d_X^{-n_0}}$. Consider the left exact
sequence $0\rightarrow {\rm Ker}d_X^{-n_0}\stackrel{\rm inc}
\rightarrow X^{-n_0} \rightarrow {\rm Ker}d_X^{-n_0+1} \rightarrow
0$. Since the right side morphism is a right
$\mathcal{X}$-approximation, it is necessarily epic and then the
sequence is exact. Because the morphism ``${\rm inc}$" admits a
retraction, the sequence splits and then ${\rm Ker}d_X^{-n_0+1}$ is
a direct summand of $X^{-n_0}$. Recall that the subcategory
$\mathcal{X}\subseteq \mathcal{A}$ is closed under taking direct
summands. Therefore the object ${\rm Ker}d_X^{-n_0+1}$ belongs to
$\mathcal{X}$.

\vskip 3pt

 The implication ``$(3) \Rightarrow (1)$" is easy, since
the subcomplex $0 \rightarrow {\rm Ker}d_X^{-n_0+1} \rightarrow
X^{-n_0+1} \rightarrow \cdots \rightarrow X^0 \rightarrow
M\rightarrow 0$ is the required $\mathcal{X}$-resolution.
\end{proof}

We have the following consequence.

\begin{cor}\label{cor:finitedimension}
Let $(\mathcal{X}, \mathcal{Y})$ be an admissible balanced pair in
an abelian category $\mathcal{A}$. Then we have
$\mathcal{X}\mbox{-res.dim}\;
\mathcal{A}=\mathcal{Y}\mbox{-cores.dim}\; \mathcal{A}$.
\end{cor}

 \begin{proof}  We apply Lemma \ref{lem:balanced}. Then the result
 follows directly from Lemma \ref{homologydimension}(2) and its dual for coadmissible covariantly
 finite subcategories.
 \end{proof}

In what follows we introduce the notion of cotorsion triple, which
gives rise naturally to a balanced pair. The notion was suggested by
Edgar Enochs in a private communication.

\vskip 5pt

 Let $\mathcal{A}$ be an
abelian category. For a subcategory $\mathcal{X}$ of $\mathcal{A}$,
set $\mathcal{X}^\perp=\{M \in \mathcal{A} \; |\; {\rm
Ext}_\mathcal{A}^1(X, M)=0 \mbox{ for all } X \in \mathcal{X}\}$ and
$^\perp \mathcal{X} = \{M \in \mathcal{A} \; |\; {\rm
Ext}_\mathcal{A}^1(M, X)=0 \mbox{ for all } X \in \mathcal{X}\}$. A
pair  $(\mathcal{X}, \mathcal{Y})$ of subcategories in $\mathcal{A}$
is called a \emph{cotorsion pair} provided that $\mathcal{X}={^\perp
\mathcal{Y}}$ and $\mathcal{Y}=\mathcal{X}^\perp$. The cotorsion
pair $(\mathcal{X}, \mathcal{Y})$ is said to be \emph{complete}
provided that for each $M\in \mathcal{A}$ there exist short exact
sequences $0\rightarrow Y \rightarrow X \rightarrow M \rightarrow 0$
and $0\rightarrow M \rightarrow Y' \rightarrow X'\rightarrow 0$ with
$X, X'\in \mathcal{X}$ and $Y, Y'\in \mathcal{Y}$ (\cite[Chapter
7]{EJ} and \cite{Hov,EEG}).

\vskip 5pt

Assume that $\mathcal{A}$ has enough projective and injective
objects. Recall that a subcategory $\mathcal{X}$ of $\mathcal{A}$ is
\emph{resolving} provided that it contains all projective objects
such that for any short exact sequence $0\rightarrow X' \rightarrow
X \rightarrow X''\rightarrow 0$ with $X''\in \mathcal{X}$ in
$\mathcal{A}$, $X\in \mathcal{X}$ if and only if  $X'\in
\mathcal{X}$. Dually one has the notion of \emph{coresolving
subcategory}. A cotorsion pair $(\mathcal{X}, \mathcal{Y})$ is said
to be \emph{hereditary} provided that $\mathcal{X}$ is resolving. It
is not hard to see that this is equivalent to that the subcategory
$\mathcal{Y}$ is coresolving (\cite[Theorem 3.4]{EJTX}).

\vskip 5pt

 A triple $(\mathcal{X}, \mathcal{Z}, \mathcal{Y})$ of
subcategories in $\mathcal{A}$ is called a \emph{cotorsion triple}
provided that both $(\mathcal{X}, \mathcal{Z})$ and $(\mathcal{Z},
\mathcal{Y})$ are cotorsion pairs; it is \emph{complete} (resp.
\emph{hereditary}) provided that both of the two cotorsion pairs are
complete (resp. hereditary).

\vskip 5pt

 The following result is essentially due to Enochs, Jenda,
Torrecillas and Xu (\cite[Theorem 4.1]{EJTX}). The argument
resembles the one in \cite[Theorem 12.1.4]{EJ}. For completeness we
include a proof.

\begin{prop}\label{prop:blanced}
Let $\mathcal{A}$ be an abelian category with enough projective and
injective objects. Assume that $(\mathcal{X}, \mathcal{Z},
\mathcal{Y})$ is  cotorsion triple which is complete and hereditary.
Then the pair $(\mathcal{X}, \mathcal{Y})$ is an admissible balanced
pair.
\end{prop}

\begin{proof} Let $M\in \mathcal{A}$. Since $(\mathcal{X}, \mathcal{Z})$ is
complete, we have a short exact sequence $\xi\colon 0\rightarrow Z
\rightarrow X \stackrel{f}\rightarrow M \rightarrow 0$ with $X\in
\mathcal{X}$ and $Z\in \mathcal{Z}$. Since $Z\in \mathcal{X}^\perp$,
the sequence $\xi$ is a \emph{special right
$\mathcal{X}$-approximation} (\cite[Definition 7.1.6]{EJ}). In
particular, we have that $f$ is a right $\mathcal{X}$-approximation,
and then $\mathcal{X}$ is contravariantly finite. Dually we have
that $\mathcal{Y}$ is covariantly finite. Then we get (BP0).

\vskip 3pt

 Observe that the subcategory $\mathcal{X}$ contains all the
projective objects. Then right $\mathcal{X}$-approximations are
epic, that is, the contravariantly finite subcategory
$\mathcal{X}\subseteq \mathcal{A}$ is admissible. Dually the
subcategory $\mathcal{Y}$ is coadmissible.

\vskip 3pt

 To show (BP1), let $X^\bullet \stackrel{\varepsilon}\rightarrow M$ be an $\mathcal{X}$-resolution of
 an object $M$. Since $\mathcal{X}$ is admissible, the sequence
 $X^\bullet \stackrel{\varepsilon}\rightarrow M$ is acyclic.
 Since $(\mathcal{X}, \mathcal{Z})$ is complete, we may assume
 that all the cocycles of $X^\bullet$ (but $M$) lie in $\mathcal{Z}$. Then
 (BP1) follows immediately from the following fact:
 for a short exact sequence $\gamma\colon 0 \rightarrow Z^0 \rightarrow X^0
\rightarrow M \rightarrow 0$ with $Z^0\in \mathcal{Z}$ and $X^0\in
\mathcal{X}$ and an object $Y\in \mathcal{Y}$, the functor ${\rm
Hom}_\mathcal{A}(-, Y)$ keeps $\gamma$ exact. The fact is equivalent
to that the induced map ${\rm Hom}_\mathcal{A}(X^0, Y) \rightarrow
{\rm Hom}_\mathcal{A}(Z^0, Y)$ is surjective. To see this, take a
short exact sequence $0\rightarrow Z_0\rightarrow I \rightarrow
Z'\rightarrow 0$ with $I$ injective. Since
$(\mathcal{X},\mathcal{Z})$ is hereditary, $\mathcal{Z}$ is
coresolving. Note that $Z_0, I\in \mathcal{Z}$, and then we have
$Z'\in \mathcal{Z}$. Observe the following commutative diagram
\[\xymatrix{
0 \ar[r] & Z^0 \ar[r] \ar@{=}[d] & X^0 \ar[r] \ar@{.>}[d] & M \ar[r]
\ar@{.>}[d] & 0\\
0\ar[r] &  Z^0 \ar[r]  & I \ar[r]  &Z' \ar[r] &0. }\]
 Since ${\rm
Ext}_\mathcal{A}^1(Z', Y)=0$, we deduce that the induced map ${\rm
Hom}_\mathcal{A}(I, Y) \rightarrow {\rm Hom}_\mathcal{A}(Z^0, Y)$ is
surjective. Note that from the commutative diagram above we infer
that the map ${\rm Hom}_\mathcal{A}(I, Y) \rightarrow {\rm
Hom}_\mathcal{A}(Z^0, Y)$ factors as ${\rm Hom}_\mathcal{A}(I, Y)
\rightarrow {\rm Hom}_\mathcal{A}(X^0, Y)\rightarrow {\rm
Hom}_\mathcal{A}(Z^0, Y)$. Therefore the map ${\rm
Hom}_\mathcal{A}(X^0, Y) \rightarrow {\rm Hom}_\mathcal{A}(Z^0, Y)$
is surjective. Dually we have (BP2).
\end{proof}

\section{Relative Derived Category}
In this section we make preparations to prove Theorem A. We
introduce the notion of relative derived category and study its
relation with homotopy categories.

\vskip 5pt

 Let $\mathcal{A}$ be an abelian category, and let $\mathcal{X}\subseteq
\mathcal{A}$ be a contravariantly finite subcategory. Recall that
the homotopy category $\mathbf{K}(\mathcal{A})$  has a canonical
triangulated structure. Denoted  by $\mathcal{X}\mbox{-ac}$ the full
triangulated subcategory of $\mathbf{K}(\mathcal{A})$ consisting of
right $\mathcal{X}$-acyclic complexes. A chain map $f^\bullet\colon
M^\bullet \rightarrow N^\bullet$ is said to be a \emph{right
$\mathcal{X}$-quasi-isomorphism} provided that for each $X\in
\mathcal{X}$, the resulting chain map ${\rm Hom}_\mathcal{A}(X,
f^\bullet)\colon  {\rm Hom}_\mathcal{A}(X, M^\bullet)\rightarrow
{\rm Hom}_\mathcal{A}(X, N^\bullet)$ is a quasi-isomorphism. Denote
by $\Sigma_\mathcal{X}$ the class of all the right
$\mathcal{X}$-quasi-isomorphisms in $\mathbf{K}(\mathcal{A})$. Note
that the class $\Sigma_\mathcal{X}$ is a saturated multiplicative
system corresponding to the subcategory $\mathcal{X}\mbox{-ac}$ in
the sense that a chain map $f^\bullet\colon M^\bullet \rightarrow
N^\bullet$ is a right $\mathcal{X}$-quasi-isomorphism if and only if
its mapping cone ${\rm Cone}(f^\bullet)$ is right
$\mathcal{X}$-acyclic (for the correspondence, consult \cite[Chapter
V, Theorem 1.10.2]{GM1}).

\begin{defn}\label{defn:relativedc}
The \emph{relative derived category}
$\mathbf{D}_\mathcal{X}(\mathcal{A})$ of $\mathcal{A}$ with respect
to $\mathcal{X}$ is defined to be the Verdier quotient (\cite{V1}
and \cite[Chapter 2]{Ne}) of $\mathbf{K}(\mathcal{A})$ modulo the
subcategory $\mathcal{X}\mbox{-ac}$, that is,
$$\mathbf{D}_\mathcal{X}(\mathcal{A}):= \mathbf{K}(\mathcal{A})/{\mathcal{X}\mbox{-ac}}
=\Sigma_\mathcal{X}^{-1}\mathbf{K}(\mathcal{A}).$$ We denote by $Q:
\mathbf{K}(\mathcal{A}) \rightarrow \mathbf{D}_\mathcal{X}(A)$ the
quotient functor.
\end{defn}

\vskip 5pt

\begin{rem}
Denote by $\mathcal{E}_\mathcal{X}$ the class of short exact
sequences in $\mathcal{A}$ on which the functors ${\rm
Hom}_\mathcal{A}(X, -)$ are exact for all $X\in \mathcal{X}$. Then
$(\mathcal{A}, \mathcal{E}_\mathcal{X})$ is an exact category in the
sense of Quillen (\cite[Appendix A]{Ke3}). Observe that if the
subcategory $\mathcal{X}$ is admissible, then the relative derived
category $\mathbf{D}_\mathcal{X}(\mathcal{A})$ coincides with
Neeman's derived category of the exact category $(\mathcal{A},
\mathcal{E}_\mathcal{X})$ (\cite[Construction 1.5]{Ne1}; also see
\cite[Sections 11 and 12]{Ke2}). Note that Buan considers relative
derived categories in quite a different setup (\cite[Section
2]{Bu}), and Gorenstein derived categories in the sense of Gao and
Zhang are examples of  relative derived categories (\cite{GZ}).
\hfill $\square$
\end{rem}

 In what follows we will study for a complex $M^\bullet$ its \emph{$\mathcal{X}$-resolution},
 that is, a right $\mathcal{X}$-quasi-isomorphism $X^\bullet \rightarrow
M^\bullet$ with each $X^i$ lying in $\mathcal{X}$. From now on,
$\mathcal{X}\subseteq \mathcal{A}$ is  a contravariantly finite
subcategory such that $\mathcal{X}\mbox{-res.dim}\; \mathcal{A} <
\infty$. Let $M^\bullet=(M^n, d_M^n)_{n\in \mathbb{Z}}$ be a complex
in $\mathcal{A}$. For each $M^n$, take a finite
$\mathcal{X}$-resolution $X^{n,
\bullet}\stackrel{\varepsilon^n}\rightarrow M^n$, where
$X^{n,\bullet} = (X^{n,-i}, d_0^{n,-i})_{i\geq 0}$. By a version of
Comparison Theorem, there exists a chain map $d_v^{n,\bullet}\colon
X^{n, \bullet} \rightarrow X^{n+1, \bullet}$ extending the map
$d_M^n\colon M^n \rightarrow M^{n+1}$. Set $d_{1}^{i,
j}=(-1)^{j}d_v^{i,j}$ for all $i, j\in \mathbb{Z}$.

\vskip 5pt

 The following argument resembles the one in \cite[Proposition 2.6]{Ric}, while it differs
 from the proof  of  \cite[Chapter XVII, Proposition 1.2]{CE}. It
 seems that the argument in \cite{CE} does not extend to our situation.

\vskip 5pt

 Consider the bigraded objects  $X^{\bullet, \bullet}$. Note that $X^{i,j}\neq
 0$ only if $-(\mathcal{X}\mbox{-res.dim}\; \mathcal{A}) \leq j\leq 0$.
 The bigraded objects $X^{\bullet, \bullet}$ are endowed with two
 endomorphisms $d_0$ and $d_1$ of degree $(0,1)$ and $(1, 0)$,
 respectively, subject to the relations $d_0\circ d_0=0$ and $d_0\circ d_1+d_1\circ d_0=0$.
 Unfortunately, $d_1\circ d_1$ is not necessarily zero.

 \vskip 3pt

 Consider the chain map $d_1^{n+1, \bullet}\circ d_1^{n, \bullet}\colon X^{n, \bullet}
 \rightarrow X^{n+2, \bullet}$, which extends the map $0=d_M^{n+1}\circ d_M^n\colon M^{n}
 \rightarrow M^{n+2}$. By a version of Comparison Theorem,
 we infer that the chain map $d_1^{n+1, \bullet}\circ d_1^{n,
 \bullet}$ is homotopic to zero. Thus the homotopy maps give rise to an endomorphism
 $d_2$ on $X^{\bullet, \bullet}$ of degree $(2, -1)$, such that $d_0\circ d_2 + d_1\circ d_1 + d_2\circ
 d_0=0$. It is a pleasant exercise to check that $d_1\circ d_2+ d_2\circ
 d_1$ commutes with $d_0$, in other words,
 $$d_1^{n+2, \bullet-1}\circ d_2^{n,\bullet}+
 d_2^{n+1, \bullet}\circ d_1^{n, \bullet}\colon X^{n, \bullet}\longrightarrow X^{n+3,
 \bullet}(-1)$$
 is a chain map, where $(-1)$ denotes the inverse of the degree-shift functor on
 complexes (see Section 1 for the notation).

\vskip 5pt

 We need the following easy lemma whose proof is routine.

 \begin{lem}
 Let $M_1, M_2$ be two objects with $\mathcal{X}$-resolutions
 $X_1^\bullet \rightarrow M_1$ and $X_2^\bullet \rightarrow
 M_2$. Let $r\geq 1$. Then any chain map $f^\bullet \colon X_1^\bullet \rightarrow
 X_2^\bullet(-r)$ is homotopic to zero. \hfill $\square$
 \end{lem}

\vskip 5pt

By the lemma above we deduce that the chain map $d_1^{n+2,
\bullet-1}\circ d_2^{n,\bullet}+
 d_2^{n+1, \bullet}\circ d_1^{n, \bullet}$ is homotopic to zero.
 Note that the homotopy maps give rise to an endomorphism $d_3$ of degree $(3,
-2)$ such
 that
 $$d_0\circ d_3 + d_1\circ d_2 + d_2\circ d_1+d_3\circ d_0=0.$$

Iterating this process of finding homotopy maps, we  construct for
each $l\geq 0$, an endomorphism $d_l$ on $X^{\bullet, \bullet}$ of
degree $(l, -l+1)$ such that $\sum_{l=0}^n d_l\circ d_{n-l}=0$
(consult the proof of \cite[Proposition 2.6]{Ric}). We will refer to
the bigraded objects $X^{\bullet,\bullet}$ together with such
endomorphisms $d_l$ as a \emph{quasi-bicomplex} in $\mathcal{A}$.

\vskip 5pt

 The ``total complex" $T^\bullet={\rm tot}(X^{\bullet,
\bullet})$  of the quasi-bicomplex $X^{\bullet, \bullet}$ is defined
as follows: $T^n:=\bigoplus_{i+j=n}X^{i, j}$ (note that this is a
finite coproduct), and the differential $d_T^n\colon T^n \rightarrow
T^{n+1}$ is defined to be $\sum_{l\geq 0}d_l$ (again this is a
finite coproduct), that is, the restriction of $d_T^n$ on $X^{i, j}$
is given by $\sum_{l\geq 0} d_l^{i, j}$. Then we infer from above
that $d_T^{n+1}\circ d_T^{n}=0$. There is a natural chain map
$\varepsilon^\bullet\colon T^\bullet \rightarrow M^\bullet$ such
that its restriction on $X^{n, 0}$ is $\varepsilon^n$ for each $n$,
and zero elsewhere.

\vskip 5pt

We have the following key observation.

\begin{prop}\label{keyobservation}
The chain map $\varepsilon^\bullet \colon T^\bullet \rightarrow
M^\bullet$ is a right $\mathcal{X}$-quasi-isomorphism; moreover, it
is a right $\mathbf{C}(\mathcal{X})$-approximation of $M^\bullet$ in
the category $\mathbf{C}(\mathcal{A})$ of complexes.
\end{prop}

\begin{proof} First we introduce a new quasi-bicomplex
$(C^{\bullet,\bullet},d_l)$ as follows: $C^{i,j}=X^{i, j}$, $j\leq
0$ and  $C^{i,1}=M^i$, and zero elsewhere; the endomorphisms $d_l$
on $C^{i,j}$ are the same as the ones on $X^{\bullet, \bullet}$ for
$j\geq 1$ or $j=0$ and $l\geq 1$; $d_0^{i,0}=\varepsilon^i$, and
$d_l$ vanishes on $C^{i, 1}$ for all $l\neq 1$, and
$d_1^{i,1}=-d_M^i$. One checks that $C^{\bullet,\bullet}$ is a
quasi-bicomplex; moreover, it is easy to see that the ``total
complex" ${\rm tot}(C^{\bullet, \bullet})$ of $C^{\bullet,\bullet}$
is the mapping cone of the chain map $\varepsilon^\bullet\colon
T^\bullet \rightarrow M^\bullet$ shifted by minus one. Then for the
first statement, it suffices to show that the complex ${\rm
tot}(C^{\bullet,\bullet})$ is right $\mathcal{X}$-acyclic.

\vskip 3pt

 Assume that $X\in \mathcal{X}$. Consider the complex $K^\bullet={\rm Hom}_\mathcal{A}(X, {\rm
tot}(C^{\bullet,\bullet}))$ of abelian groups. Observe that the
complex $K^\bullet$ is the ``total complex" of the quasi-bicomplex
${\rm Hom}_\mathcal{A}(X, C^{\bullet,\bullet})$ of abelian groups.
As in the case of bicomplexes, we have a descending filtration of
subcomplexes $\{F^pK^\bullet, \; p\in \mathbb{Z}\}$ of the ``total
complex" $K^\bullet$ given by $F^pK^n:=\bigoplus_{i\geq p,\;
i+j=n}{\rm Hom}_\mathcal{A}(X, C^{i, j})$. This filtration gives
rise to a convergent spectral sequence $E_2^{p, q}
\underset{p}\Longrightarrow H^{p+q}(K^\bullet)$. Since $X^{n,
\bullet}\stackrel{\varepsilon^n}\rightarrow M^n$ is an
$\mathcal{X}$-resolution, the complex ${\rm Hom}_\mathcal{A}(X,
C^{n, \bullet})$ is acyclic for each  $n$. Therefore the spectral
sequence vanishes on $E_2$ (and even on $E_1$), and then we deduce
that $H^n(K^\bullet)=0$ for each $n$. We are done with the first
statement.

\vskip 3pt

For the second statement, let $f^\bullet\colon X^\bullet \rightarrow
M^\bullet$ be a chain map with $X^\bullet=(X^n, d_X^n)_{n\in
\mathbb{Z}} \in \mathbf{C}(\mathcal{X})$. Note that the morphism
$\varepsilon^n\colon X^{n, 0}\rightarrow M^n$ is a right
$\mathcal{X}$-approximation, hence the map $f^n$ factors through it.
Take $f_0^n\colon X^n\rightarrow X^{n, 0}$ such that
$\varepsilon^n\circ f_0^n=f^n$. Consider the map $d_1^{n,0}\circ
f_0^n-f_0^{n+1}\circ d_X^n\colon X^{n} \rightarrow X^{n+1,0}$. Note
that
\begin{align*}
&\varepsilon^{n+1}\circ (d_1^{n,0}\circ f_0^n-f_0^{n+1}\circ d_X^n)\\
&\;  = d_M^n \circ \varepsilon^n \circ f_0^n -
\varepsilon^{n+1}\circ
f_0^{n+1}\circ d_X^n\\
&\;  = d_M^n\circ f^n- f^{n+1}\circ  d_X^n=0.
\end{align*}
Therefore the map $d_1^{n,0}\circ f_0^n-f_0^{n+1}\circ d_X^n$
factors through ${\rm Ker}\varepsilon^{n+1}$. Note that $X^{n+1,-1}
\stackrel{d_0^{n+1,-1}} \rightarrow {\rm Ker}\varepsilon^{n+1}$ is a
right $\mathcal{X}$-approximation. Then we have a factorization
\begin{align} \label{eqnf1}
d_1^{n,0}\circ
f_0^n-f_0^{n+1}\circ d_X^n=-d_0^{n+1,-1}\circ f_1^n,
\end{align}
where $f_1^n\colon X^n \rightarrow X^{n+1, -1}$ is some morphism.

\vskip 3pt

Rewrite  equation (\ref{eqnf1}) as $d_0\circ f_1+d_1\circ
f_0=f_0\circ d_X$.  We will refer to (\ref{eqnf1})  as the
\emph{defining identity} for  $f_1^\bullet$. We claim that there
exist morphisms (not chain maps) $f_l^\bullet \colon
X^\bullet\rightarrow X^{\bullet+l,-l}$ such that $\sum_{i=0}^l
d_i\circ f_{l-i}=f_{l-1} \circ d_X$ for all $l\geq 1$. Assume that
the required $f_1, \ldots ,f_l$ are chosen. The following
computation is similar to the one in the proof of \cite[Propositions
2.6 and 2.7]{Ric}.
\begin{align*}
& d_0 \circ (\sum_{1\leq i\leq l+1}d_i\circ f_{l+1-i}-f_l\circ d_X)\\
&\; = \sum_{1\leq i\leq l+1} (d_0\circ d_i)\circ f_{l+1-i}- d_0\circ f_l\circ d_X\\
&\; = \sum_{1\leq i\leq l+1} (-\sum_{1\leq j\leq i}d_j\circ d_{i-j})\circ f_{l+1-i}-d_0\circ f_l\circ d_X\\
&\; =-\sum_{1\leq j\leq l} d_j\circ (\sum_{0\leq i\leq l-j}d_i\circ
f_{l+1-j-i})-d_{l+1}\circ d_0\circ f_0 - d_0\circ f_l\circ d_X\\
&\;= -\sum_{1\leq j\leq l}d_j\circ(f_{l-j}\circ d_X)- d_0\circ
f_l\circ d_X\\
&\; =-\sum_{0\leq j\leq l} (d_j\circ f_{l-j}) \circ d_X\\
&\; =-f_{l-1}\circ d_X\circ d_X=0.
\end{align*}
Note that the second equality uses the identities on the
endomorphisms $d_l$'s; the fourth one uses the fact $d_0\circ f_0=0$
(note that $X^{n,1}=0$) and the defining identity for
$f_{l+1-j}^\bullet$; the sixth uses the defining identity for
$f_{l}^\bullet$. We infer that the morphism
$$\sum_{1\leq i\leq l+1}d_i^{n+l+1-i, -l-1+i}\circ
f_{l+1-i}^{n}-f_l^{n+1}\circ d_X^n\colon X^n\longrightarrow
X^{n+1+l,-l}$$ factors through ${\rm Ker}d_0^{n+1+l,-l}$ and then
factors through $X^{n+1+l, -l-1}$, since the induced  map $X^{n+1+l,
-l-1}\rightarrow {\rm Ker}d_0^{n+1+l,-l}$ is a right
$\mathcal{X}$-approximation. Take $f_{l+1}^n\colon X^n \rightarrow
X^{n+1+l,-l-1}$ to fulfil the factorization. This completes the
construction of $f_{l+1}^\bullet$'s and by induction we construct
all the $f_l^\bullet$'s.

\vskip3pt

 Consider the map $\sum_{i\geq
0}f_i^n\colon X^n \rightarrow T^n=\bigoplus_{i\geq 0}X^{n+i, -i}$.
One checks readily that this defines a chain map from $X^\bullet$ to
$T^\bullet$; moreover, this chain map makes $f^\bullet$ factor
through  $\varepsilon^\bullet$. This proves that
$\varepsilon^\bullet$ is a right
$\mathbf{C}(\mathcal{X})$-approximation of $M^\bullet$.
\end{proof}

\vskip 5 pt

The following result is a relative version of a well-known result
(\cite[p.439, Proposition 2.12]{Iv}).

\begin{prop}\label{keyproposition}
Let $\mathcal{X}\subseteq \mathcal{A}$ be a contravariantly finite
subcategory. Assume that $\mathcal{X}$ is admissible and
$\mathcal{X}\mbox{-res.dim}\; \mathcal{A}< \infty$. Then the natural
composite functor $\mathbf{K}(\mathcal{X}) \stackrel{\rm
inc}\rightarrow \mathbf{K}(\mathcal{A}) \stackrel{Q}\rightarrow
\mathbf{D}_\mathcal{X}(\mathcal{A})$ is a triangle-equivalence.
\end{prop}

\vskip 3pt

\begin{proof} The composite functor is clearly a triangle functor. It
suffices to show it is an equivalence of categories (see
\cite[p.4]{Ha1}). By Proposition \ref{keyobservation} for each
complex $M^\bullet$, there is an $\mathcal{X}$-resolution
$\varepsilon^\bullet\colon X^\bullet \rightarrow M^\bullet$, that
is, it is a right $\mathcal{X}$-quasi-isomorphism. Note that
$\varepsilon^\bullet$ becomes an isomorphism in the relative derived
category $\mathbf{D}_\mathcal{X}(\mathcal{A})$, in particular,
$Q(M^\bullet) \simeq Q\circ {\rm inc}(X^\bullet)$. Therefore the
composite functor is dense.

\vskip 3pt

 We claim that for each $X_0^\bullet\in
 \mathbf{K}(\mathcal{X})$ and each
 right $\mathcal{X}$-acyclic complex $M^\bullet\in \mathbf{K}(\mathcal{A})$, ${\rm Hom}_{\mathbf{K}(\mathcal{A})}
(X_0^\bullet, M^\bullet)=0$. This will complete the proof by the
following general fact: for a triangulated category $\mathcal{T}$
and a triangulated subcategory  $\mathcal{N}\subseteq \mathcal{T}$,
set  $^\perp\mathcal{N}=\{X\in \mathcal{T}\; |\; {\rm
Hom}_\mathcal{T}(X, N)=0 \mbox{ for all } N\in \mathcal{N}\}$ to be
the \emph{left perpendicular subcategory}, then the composite
functor $^\perp\mathcal{N} \stackrel{\rm inc} \rightarrow
\mathcal{T} \stackrel{Q}\rightarrow \mathcal{T}/\mathcal{N}$ is
fully faithful (\cite[5-3 Proposition]{V1}). The claim says
precisely that $\mathbf{K}(\mathcal{X}) \subseteq
{^\perp(\mathcal{X}\mbox{-ac})}$. By the recalled general fact the
composite functor is fully faithful. Note that the functor is dense
by above, thus it is an equivalence of categories. \vskip 3pt

 To see the
claim, take a chain map $f^\bullet\colon X_0^\bullet \rightarrow
M^\bullet$. By Proposition \ref{keyobservation} we may take an
$\mathcal{X}$-resolution $\varepsilon^\bullet\colon X^\bullet
\rightarrow M^\bullet$ which is a right
$\mathbf{C}(\mathcal{X})$-approximation. Hence $f^\bullet$ factors
through $\varepsilon^\bullet$. In fact, we will show that
$X^\bullet$ is null-homotopic, and then $\varepsilon^\bullet$ and
consequently $f^\bullet$ is homotopic to zero. Set
$\mathcal{X}\mbox{-res.dim}\; \mathcal{A}=n_0$. Note that
$X^\bullet=(X^n, d_X^n)_{n\in \mathbb{Z}}$ is right
$\mathcal{X}$-acyclic. Consider the canonical factorization $X^n
\stackrel{\partial^n}\rightarrow {\rm Ker}d_X^{n+1} \stackrel{\rm
inc} \rightarrow X^{n+1}$ of the differential $d_X^n$. Recall that
the complex $X^\bullet$ is null-homotopic if and only if the
morphisms $\partial^n$ are split epic. Note that the subcomplex
$\cdots \rightarrow X^{n-1}\rightarrow
X^n\stackrel{\partial^n}\rightarrow {\rm Ker}d_X^{n+1} \rightarrow
0$ can be viewed as a shifted $\mathcal{X}$-resolution. By Lemma
\ref{homologydimension}(3) we have that ${\rm Ker}d_X^{n-n_0+1}$
belongs to $\mathcal{X}$. Thus all the cocycles ${\rm Ker}d_X^n$ of
$X^\bullet$ lie in $\mathcal{X}$. Since $X^\bullet$ is right
$\mathcal{X}$-acyclic, the morphism $\partial^{n}\colon X^n
\rightarrow {\rm Ker}d_X^{n+1}$ is a right
$\mathcal{X}$-approximation. In particular, the identity map of
${\rm Ker}d_X^{n+1}$ factors through $\partial^{n}$, that is, the
morphism $\partial^n$ is split epic. We are done.
\end{proof}

\begin{rem}\label{construction}
The composite functor in Proposition \ref{keyproposition} factors as
$$\mathbf{K}(\mathcal{X})\stackrel{\rm inc}\longrightarrow
{^\perp(\mathcal{X}\mbox{-ac})} \stackrel{\rm inc}\longrightarrow
\mathbf{K}(\mathcal{A}) \stackrel{Q}\longrightarrow
\mathbf{D}_\mathcal{X}(\mathcal{A}).$$ By the recalled general fact,
the composite of the latter two functors is  fully faithful. Hence
the equivalence in Proposition \ref{keyproposition} will force the
equality $\mathbf{K}(\mathcal{X})={^\perp(\mathcal{X}\mbox{-ac})}$,
and it also implies  that the subcategory
$\mathcal{X}\mbox{-ac}\subseteq \mathbf{K}(\mathcal{A})$ is
\emph{left admissible} and $\mathbf{K}(\mathcal{X})\subseteq
\mathbf{K}(\mathcal{A})$ is \emph{right admissible} (=
\emph{Bousfield}) (\cite[Definition 1.2]{BK}; compare \cite[Chapter
9]{Ne}). Hence the inclusion functor ${\rm inc}:
\mathbf{K}(\mathcal{X}) \rightarrow \mathbf{K}(\mathcal{A})$ has  a
right adjoint $i^!\colon  \mathbf{K}(\mathcal{A}) \rightarrow
\mathbf{K}(\mathcal{X})$. The functor $i^!$ vanishes on
$\mathcal{X}\mbox{-ac}$ and then factors through the quotient
functor $Q\colon \mathbf{K}(\mathcal{A}) \rightarrow
\mathbf{D}_\mathcal{X}(\mathcal{A})$ canonically; by abuse of
notation we denote the resulting functor by $i^!\colon
\mathbf{D}_\mathcal{X}(\mathcal{A}) \rightarrow
\mathbf{K}(\mathcal{X})$. This functor is a quasi-inverse of the
composite functor in the Proposition \ref{keyproposition}. \vskip
5pt

 For later use, let us recall
the construction of the quasi-inverse functor $i^!$: for each
complex $M^\bullet$, choose a complex $i^!(M^\bullet) \in
\mathbf{K}(\mathcal{X})$ and fix a right
$\mathcal{X}$-quasi-isomorphism $\varepsilon^\bullet\colon
i^!(M^\bullet)\rightarrow M^\bullet$; for a chain map
$f^\bullet\colon M^\bullet \rightarrow M'^\bullet$, there is a
unique, up to homotopy, chain map $i^!(f^\bullet)\colon
i^!(M^\bullet) \rightarrow i^!(M'^\bullet)$ making the following
diagram commute, again up to homotopy

\[\xymatrix{
i^!(M^\bullet)\ar[r]^-{\varepsilon^\bullet}
\ar@{.>}[d]_{i^!(f^\bullet)}  & M^\bullet \ar[d]^-{f^\bullet}\\
i^!(M'^\bullet) \ar[r]^-{\varepsilon'^\bullet} & M'^\bullet }\] One
could deduce this by Proposition \ref{keyproposition}, or
alternatively, by noting that the cohomological functor ${\rm
Hom}_{\mathbf{K}(\mathcal{A})}(i^!(M^\bullet), -)$ vanishes on the
mapping cone of $\varepsilon'^\bullet$, and then one gets the
natural isomorphism
$${\rm Hom}_{\mathbf{K}(\mathcal{A})}(i^!(M^\bullet), i^!(M'^\bullet))\simeq {\rm
Hom}_{\mathbf{K}(\mathcal{A})}(i^!(M^\bullet), M'^\bullet).$$
 In this way one defines the functor $i^!$ on homotopy categories, which
 induces the pursued functor $i^!\colon \mathbf{D}_\mathcal{X}(\mathcal{A}) \rightarrow
\mathbf{K}(\mathcal{X}) $. \hfill $\square$
\end{rem}

\section{Proof of Theorem A}

In this section we prove Theorem A.

\vskip 5pt

Let $\mathcal{A}$ be an abelian category. Let $(\mathcal{X},
\mathcal{Y})$ be an admissible balanced pair in $\mathcal{A}$ of
finite dimension. For each complex $X^\bullet \in
\mathbf{K}(\mathcal{X})$, choose a complex $F(X^\bullet)\in
\mathbf{K}(\mathcal{Y})$ and fix a left
$\mathcal{Y}$-quasi-isomorphism $X^\bullet
\stackrel{\theta_X^\bullet}\rightarrow F(X^\bullet)$ (see
Proposition \ref{keyobservation}); for each chain map
$f^\bullet\colon X^\bullet\rightarrow X'^\bullet$ there is a unique,
up to homotopy, chain map $F(f^\bullet)\colon F(X^\bullet)
\rightarrow F(X'^\bullet)$ such that $F(f^\bullet)\circ
\theta_X^\bullet=\theta_{X'}^\bullet \circ f^\bullet$, again up to
homotopy; see Proposition \ref{keyproposition}. This defines a
triangle functor $F\colon \mathbf{K}(\mathcal{X}) \rightarrow
\mathbf{K}(\mathcal{Y})$.

\begin{thm}\label{thm:theoremA}
Let $\mathcal{A}$ be an abelian category. Let $(\mathcal{X},
\mathcal{Y})$ be a balanced pair of subcategories in $\mathcal{A}$
which is admissible and of finite dimension. Then the above defined
triangle functor $F\colon \mathbf{K}(\mathcal{X}) \simeq
\mathbf{K}(\mathcal{Y})$ is an equivalence.
\end{thm}

\begin{proof} Note that by Proposition \ref{firstprop}, the full
subcategories $\mathcal{X}\mbox{-ac}=\mathcal{Y}\mbox{-ac}$. Here
$\mathcal{Y}\mbox{-ac}$ means the full subcategory of
$\mathbf{K}(\mathcal{A})$ consisting of left $\mathcal{Y}$-acyclic
complexes. Then we have
$\mathbf{D}_\mathcal{X}(\mathcal{A})=\mathbf{D}_\mathcal{Y}(\mathcal{A})$.
Applying the dual of Proposition \ref{keyproposition} to
$\mathcal{Y}$, we get a natural triangle-equivalence
$\mathbf{K}(\mathcal{Y}) \stackrel{\sim}\longrightarrow
\mathbf{D}_{\mathcal{Y}}(\mathcal{A})$. Composing a quasi-inverse of
this equivalence with the equivalence in Proposition
\ref{keyproposition}, we get a triangle-equivalence $F'\colon
\mathbf{K}(\mathcal{X})\stackrel{\sim}\longrightarrow
\mathbf{K}(\mathcal{Y})$.

Now observe that  the triangle-equivalence $F'$ coincides with the
functor $F$ just defined above. This follows from the construction
in Remark \ref{construction}, while here we need to dualize the
argument to construct a quasi-inverse functor of the equivalence
$\mathbf{K}(\mathcal{Y})\stackrel{\sim}\longrightarrow
\mathbf{D}_\mathcal{Y}(\mathcal{A})$.
\end{proof}

\section{Proof of Theorem B}
In this section  we apply  Theorem \ref{thm:theoremA} to obtain
Theorem B. We will make use of a characterization theorem of
left-Gorenstein rings by Beligiannis (\cite{Bel}).

\vskip 5pt

Let $R$ be a ring with identity. Denote by $R\mbox{-Mod}$ the
category of left $R$-modules and  by $\mathcal{L}$ the full
subcategory consisting of modules with finite projective and
injective dimension. Following \cite{Bel} a ring $R$ is called
left-Gorenstein provided that any module in $R\mbox{-Mod}$ has
finite projective dimension if and only if it has finite injective
dimension. In this case, by \cite[Theorem 6.9($\delta$)]{Bel} there
is a uniform upper bound $d$ such that each module in $\mathcal{L}$
has projective and injective dimensions less or equal to $d$. We
will denote by ${\rm G.dim}\; R$ the minimal bound.

\vskip 5pt

We collect in the following lemma some crucial properties of
left-Gorenstein rings.

\begin{lem} \label{lem:left-Gorenstein} Let $R$ be a left-Gorenstein ring. Then we have the
following:
\begin{enumerate}
\item the triple $(R\mbox{-{\rm GProj}}, \mathcal{L}, R\mbox{-{\rm GInj}})$ is a
 complete and hereditary cotorsion triple;
\item $R\mbox{-{\rm GProj}}\mbox{-res.dim}\; R\mbox{-{\rm Mod}}={\rm G.dim}\; R= R\mbox{-{\rm GInj}}\mbox{-cores.dim}\;
R\mbox{-{\rm Mod}}$.
\end{enumerate}
\end{lem}

\begin{proof}
We infer (1) by \cite[Theorem 6.9(4) and (5)]{Bel} (which is
presented in quite a different terminology). One might deduce (1)
also from \cite[Theorems 2.25 and 2.26]{EEG}. Just note that for a
left-Gorenstein ring $R$, $R\mbox{-Mod}$ is a Gorenstein category in
the sense of Enochs, Estrada and Garc\'{\i}a Roza (\cite[Definition
2.18]{EEG}). The statement (2) follows from \cite[Theorem
6.9($\alpha$)]{Bel}.
\end{proof}

We are in the position to prove Theorem B.

\begin{thm}\label{thm:theoremB}
Let $R$  be a left-Gorenstein ring. Then we have a
triangle-equivalence $\mathbf{K}(R\mbox{-{\rm GProj}}) \simeq
\mathbf{K}(R\mbox{-{\rm GInj}})$, which restricts to a
triangle-equivalence $\mathbf{K}(R\mbox{-{\rm Proj}}) \simeq
\mathbf{K}(R\mbox{-{\rm Inj}})$.
\end{thm}

\begin{proof} Combining Lemma \ref{lem:left-Gorenstein}(1) and
Proposition \ref{prop:blanced} together, we infer that the pair
$(R\mbox{-GProj}, R\mbox{-GInj})$ is an admissible balanced pair in
$R\mbox{-Mod}$. It is of finite dimension by Lemma
\ref{lem:left-Gorenstein}(2). By Theorem \ref{thm:theoremA} we get a
triangle-equivalence $F\colon
\mathbf{K}(R\mbox{-GProj})\stackrel{\sim}\longrightarrow
\mathbf{K}(R\mbox{-GInj})$. Denote by $F^{-1}$ its quasi-inverse.
Remind that the construction of the functors $F$ and $F^{-1}$ is
described  before Theorem \ref{thm:theoremA} (and its dual).

\vskip 3pt

 Recall that for a Gorenstein projective module $G$ we
have ${\rm Ext}_R^i(G, P)=0$ for $i\geq 1$ and all projective
modules $P$ (\cite[Lemma 2.2]{CFH}). Then by the dimension-shift
technique we have ${\rm Ext}_R^i(G, M)=0$ for $i\geq 1$ and all
modules $M$ of finite projective dimension. Set $d={\rm G.dim}\; R$.
For a projective module $P$ consider its injective resolution
$0\rightarrow P\rightarrow I^0\rightarrow I^1 \rightarrow \cdots
\rightarrow I^d\rightarrow 0$. Write it as $P\rightarrow I^\bullet$.
It is obvious that for a Gorenstein projective module $G$, ${\rm
Hom}_R(G, I^\bullet)$ has no cohomology in non-zero degrees, for it
computes ${\rm Ext}^*_R(G, P)$. Thus the injective resolution is
right $R\mbox{-GProj} $-acyclic, and by Proposition \ref{firstprop},
it is also left $R\mbox{-GInj}$-acyclic. In particular, it is an
$R\mbox{-GInj}$-coresolution. Take a complex $P^\bullet$ in
$\mathbf{K}(R\mbox{-Proj})$. Consider the construction of
$R\mbox{-GInj}$-coresolution as in the dual of Proposition
\ref{keyobservation}. We find that the $R\mbox{-GInj}$-coresolution
of $P^\bullet$ is a complex consisting of injective modules. That
is, the essential image of $\mathbf{K}(R\mbox{-Proj})$ under $F$
lies in $\mathbf{K}(R\mbox{-Inj})$. Dually the essential image of
$\mathbf{K}(R\mbox{-Inj})$ under $F^{-1}$ lies in
$\mathbf{K}(R\mbox{-Proj})$. Consequently, we have a restricted
equivalence $\mathbf{K}(R\mbox{-Proj})\simeq
\mathbf{K}(R\mbox{-Inj})$.
\end{proof}

\section{Comparison of Equivalences}

In the last section we compare the equivalences in Theorem
\ref{thm:theoremB} with Iyengar-Krause's equivalence (\cite{IK}) in
the case of commutative Gorenstein rings. In this case, it turns out
that  up to a natural isomorphism the first equivalence in Theorem
\ref{thm:theoremB} extends Iyengar-Krause's equivalence.

\vskip 5pt

Let $R$ be a commutative Gorenstein ring of dimension $d$. Take its
injective resolution $0\rightarrow R
\stackrel{\varepsilon}\rightarrow I^0 \rightarrow I^1\rightarrow
\cdots \rightarrow I^d \rightarrow 0$. Write it as
$R\stackrel{\varepsilon }\rightarrow I^\bullet$. Then the complex
$I^\bullet$ is a \emph{dualizing complex}; for details, see
\cite[Chapter V, \S 2]{Har}, \cite[Section 3]{IK} and \cite[Appendix
A]{CFH}.

\vskip 5pt

Note that the ring $R$ is noetherian, and then the class of
injective modules is closed under coproducts. One infers that for a
projective module $P$ and an injective module $I$ the tensor module
$P\otimes_R I$ is injective. Then we have a well-defined triangle
functor
$$-\otimes_R I^\bullet:\;  \mathbf{K}(R\mbox{-Proj})\longrightarrow \mathbf{K}(R\mbox{-Inj}).$$
 By \cite[Theorem 4.2]{IK} this is a triangle-equivalence, which we will call
\emph{Iyengar-Krause's equivalence}.

\vskip 5pt

We note the following fact.

\begin{lem} \label{lem:tensor} {\rm (\cite[Corollary 5.7]{CH})} Let
$R$ be a commutative Gorenstein ring. Then for a Gorenstein
projective module $G$ and an injective module $I$, the tensor module
$G\otimes_R I$ is Gorenstein injective. \hfill $\square$
\end{lem}

By the lemma above we can extend Iyengar-Krause's equivalence to a
triangle functor
$$-\otimes_R I^\bullet\colon \; \mathbf{K}(R\mbox{-GProj})\longrightarrow \mathbf{K}(R\mbox{-GInj}).$$

\vskip 5pt

 Recall the construction of the equivalence
 $F\colon \mathbf{K}(R\mbox{-GProj}) \stackrel{\sim}\longrightarrow \mathbf{K}(R\mbox{-GInj})$
 in  Theorem \ref{thm:theoremB}. For
each $G^\bullet \in \mathbf{K}(R\mbox{-GProj})$ choose an
$R\mbox{-GInj}$-coresolution  $\theta_G^\bullet \colon
G^\bullet\rightarrow F(G^\bullet) $; for each $f^\bullet \colon
G^\bullet \rightarrow G'^\bullet$ there is a unique, up to homotopy,
chain map $F(f^\bullet)\colon F(G^\bullet)\rightarrow F(G'^\bullet)$
such that $F(f^\bullet)\circ \theta_G^\bullet=\theta_{G'}^\bullet
\circ f^\bullet$. This defines the triangle functor $F$. Consult the
construction before Theorem \ref{thm:theoremA}.

\vskip 5pt

 Note that the mapping cone ${\rm Cone}(\theta_G^\bullet)$
of $\theta_G^\bullet$ is left $R\mbox{-GInj}$-acyclic. By the proof
of Proposition \ref{keyproposition} we have
$${\rm
Hom}_{\mathbf{K}(R\mbox{-}{\rm Mod})}({\rm Cone}(\theta_G^\bullet),
G^\bullet \otimes_R I^\bullet[n])=0, \; \mbox{for all } n\in
\mathbb{Z}.$$ By applying the cohomological functor ${\rm
Hom}_{\mathbf{K}(R\mbox{-}{\rm Mod})}(-, G^\bullet \otimes_R
I^\bullet)$ to the distinguished triangle associated to
$\theta_G^\bullet$, we deduce a natural isomorphism of abelian
groups
$${\rm Hom}_{\mathbf{K}(R\mbox{-}{\rm Mod})}(G^\bullet, G^\bullet
\otimes_R I^\bullet)\simeq {\rm Hom}_{\mathbf{K}(R\mbox{-}{\rm
Mod})}(F(G^\bullet), G^\bullet \otimes_R I^\bullet).$$
 Note that there is a natural chain map ${\rm Id}_{G^\bullet}\otimes_R
\varepsilon\colon  G^\bullet\rightarrow G^\bullet\otimes_R
I^\bullet$. By the above isomorphism, there exists a unique, up to
homotopy, chain map $\eta_{G^\bullet}\colon  F(G^\bullet)\rightarrow
G^\bullet \otimes I^\bullet$ such that $ \eta_{G^\bullet}\circ
\theta_{G}^\bullet={\rm Id}_{G^\bullet}\otimes_R \varepsilon$. It is
routine to check that this defines a natural transformation of
triangle functors
$$\eta\colon F \; \longrightarrow -\otimes_R I^\bullet.$$

\vskip 10pt

The following result states that  in the case of commutative
Gorenstein rings the first equivalence in Theorem \ref{thm:theoremB}
extends Iyengar-Krause's equivalence, up to a natural isomorphism.

\vskip 5pt

\begin{prop} \label{prop:comparisonequivalence}  Use the notation as
above. Then for each complex $P^\bullet\in \mathbf{K}(R\mbox{-{\rm
Proj}})$, the chain map $\eta_{P^\bullet}$ is an isomorphism in
$\mathbf{K}(R\mbox{-{\rm GInj}})$.
\end{prop}

\begin{proof} First note that $\eta_{G^\bullet}$ is an isomorphism if and
only if ${\rm Id}_{G^\bullet}\otimes_R \varepsilon:\;
G^\bullet\rightarrow G^\bullet\otimes I^\bullet$ is a left
$R\mbox{-GInj}$-quasi-isomorphism. The ``only if" part is clear
since $\theta_G^\bullet$ is a coresolution. For the ``if" part,
assume that ${\rm Id}_{G^\bullet}\otimes_R \varepsilon:\;
G^\bullet\rightarrow G^\bullet\otimes I^\bullet$ is left
$R\mbox{-GInj}$-quasi-isomorphism. Then by a similar argument as
above we get a unique chain map $\gamma_{G^\bullet}\colon G^\bullet
\otimes I^\bullet \rightarrow F(G^\bullet)$ such that $
\theta_{G}^\bullet= \gamma_{G}^\bullet \circ({\rm
Id}_{G^\bullet}\otimes_R \varepsilon)$. Then these two ``uniqueness"
imply that $\eta_{G}^\bullet$ and $\gamma_{G}^\bullet$ are inverse
to each other.

\vskip 3pt

Note that ${\rm Id}_{G^\bullet}\otimes_R \varepsilon:\;
G^\bullet\rightarrow G^\bullet\otimes_R I^\bullet$ is a left
$R\mbox{-GInj}$-quasi-isomorphism if and only if its mapping cone is
left $R\mbox{-GInj}$-acyclic, or equivalently, by Proposition
\ref{firstprop}, its mapping cone is right $R\mbox{-GProj}$-acyclic.
However the mapping cone is given by the tensor complex
$G^\bullet\otimes_R Y^\bullet$; here we denote by $Y^\bullet$ the
acyclic complex $0\rightarrow R \stackrel{\varepsilon}\rightarrow
I^0 \rightarrow I^1\rightarrow \cdots \rightarrow I^d \rightarrow
0$. So it suffices to show that for each complex $P^\bullet \in
\mathbf{K}(R\mbox{-Proj})$ the tensor complex $P^\bullet\otimes_R
Y^\bullet$ is right $R\mbox{-GProj}$-acyclic.

\vskip 3pt

Given any Gorenstein projective module $G$, we need to show that the
Hom complex ${\rm Hom}_R(G, P^\bullet\otimes_R Y^\bullet)$ is
acyclic. Note that the tensor complex $P^\bullet\otimes_R Y^\bullet$
is the total complex of a bicomplex $K^{\bullet, \bullet}$ such that
$K^{i, j}=P^i\otimes_R Y^j$. Therefore one sees that the complex
${\rm Hom}_R(G, P^\bullet\otimes_R Y^\bullet)$ is  the total complex
of the bicomplex ${\rm Hom}_R(G, K^{\bullet, \bullet})$. Associated
to this bicomplex, there exists a convergent spectral sequence
$E_2^{p, q} \underset{p}\Longrightarrow H^{p+q}({\rm Hom}_R(G,
P^\bullet\otimes_R Y^\bullet))$. Note that for each $i$, the column
complex $K^{i,\bullet}$ is an injective resolution of $P^i$. By the
second paragraph in the proof of Theorem \ref{thm:theoremB}, we
infer that $K^{i,\bullet}$ is right $R\mbox{-GProj}$-acyclic. Hence
the column complex ${\rm Hom}_R(G, K^{i, \bullet})$ is acyclic for
each $i$. Therefore, in the spectral sequence we see that $E_2$ (and
even $E_1$) vanishes. Thus we get $H^{n}({\rm Hom}_R(G,
P^\bullet\otimes_R Y^\bullet))=0$ for all $n\in \mathbb{Z}$. We are
done. \end{proof}

\vskip 5pt

\noindent {\bf Acknowledgements}\quad The author is indebted to the
anonymous referee for numerous suggestions which improve the
exposition very much. This paper is completed during the author's
visit in the University of Paderborn with a support by Alexander von
Humboldt Stiftung. He would like to thank Prof. Henning Krause and
the faculty of Institut fuer Mathematik for their hospitality. The
author would like to thank Prof. Edgar Enochs for suggesting him the
notion of cotorsion triple.

\bibliography{}

\vskip 20pt

{\footnotesize \noindent Xiao-Wu Chen, Department of Mathematics,
University of Science and Technology of
China, Hefei 230026, P. R. China \\
\emph{Current address}: Institut fuer Mathematik, Universitaet
Paderborn, 33095, Paderborn, Germany}
\end{document}